\documentclass[11pt]{article}
\usepackage[a4paper]{geometry}
\usepackage[english]{babel}
\usepackage[utf8x]{inputenc}
\usepackage{graphicx}
\usepackage[colorinlistoftodos]{todonotes}

\usepackage{latexsym}        
\usepackage{graphicx}        
\usepackage{color}
\usepackage{subfigure}
\usepackage{cite}
\usepackage{amsmath,amsfonts}
\usepackage{amsthm}

\pdfoutput=1

\usepackage{mathtools}
\usepackage[subnum]{cases}

\newtheorem{remark}{Remark}

\newcommand{\beq}{\begin{equation}}
\newcommand{\eeq}{\end{equation}}
\newcommand{\bdm}{\begin{displaymath}}
\newcommand{\edm}{\end{displaymath}}
\newcommand{\beqa}{\begin{eqnarray}}
\newcommand{\eeqa}{\end{eqnarray}}
\newcommand{\beqas}{\begin{eqnarray*}}
\newcommand{\eeqas}{\end{eqnarray*}}

\newcommand{\bb}{{\bf b}}

\newcommand{\bn}{{\bf n}}

\newcommand{\bxi}{\boldsymbol\xi}
\newcommand{\bp}{{\bf p}}
\newcommand{\bq}{{\bf q}}
\newcommand{\br}{{\bf r}}
\newcommand{\bu}{{\bf u}}
\newcommand{\bv}{{\bf v}}
\newcommand{\bw}{{\bf w}}
\newcommand{\bx}{{\bf x}}
\newcommand{\by}{{\bf y}}
\newcommand{\bz}{{\bf z}}
\newcommand{\bh}{{\bf h}}
\newcommand{\bzz}{{\bf 0}}
\newcommand{\bR}{{\bf R}}

\newcommand{\bbR}{\mathbb{R}}

\newcommand{\Div}{\nabla\cdot}

\newcommand{\tr}{\mbox{\normalfont{tr}}}

\newcommand{\Ldel}{\mathcal{L}^\delta}
\newcommand{\Ls}{\mathcal{L}^\delta_s}
\newcommand{\Lz}{\mathcal{L}^\delta_{0}}
\newcommand{\N}{\mathcal{N}}
\newcommand{\Ns}{\mathcal{N}_s}
\newcommand{\Ld}{\mathcal{L}^\delta_d}
\newcommand{\Nd}{\mathcal{N}_d}
\newcommand{\limd}{\lim_{\delta\rightarrow 0}}
\newcommand{\BigO}[1]{\ensuremath{\operatorname{O}\left(#1\right)}}
\newcommand{\Bp}[1]{B^+_\delta\left(#1\right)}
\newcommand{\Bm}[1]{B^-_\delta(#1)}
\newcommand{\jump}[1]{\left[#1\right]^+_-}

\newtheorem{thm}{Theorem}
\newtheorem{cor}{Corollary}
\newtheorem{lem}{Lemma}

\newtheorem{prop}{Proposition}

\numberwithin{equation}{section}

\begin{document}
\nocite{*}
\date{}

\title{Peridynamics and Material Interfaces}
\author{Bacim Alali$^{\tiny a}$ \and Max Gunzburger$^{\tiny b}$}
\maketitle
\let\thefootnote\relax\footnote{$^{ a}$Department of Mathematics, Kansas State University, Manhattan, KS 66506.\\{\it bacimalali@math.ksu.edu.}
}
\let\thefootnote\relax\footnote{$^{ b}$  
Department of Scientific Computing, Florida State University, 
Tallahassee, FL 32306.\\{\it mgunzburger@fsu.edu.}
}


\begin{abstract}
The convergence of a peridynamic model for solid mechanics inside  heterogeneous media in the limit of vanishing nonlocality is analyzed.
It is shown that the operator of linear peridynamics for an isotropic heterogeneous medium converges
to the corresponding operator of linear elasticity when the material properties are sufficiently regular.
On the other hand, when the material properties are discontinuous, i.e., when material interfaces are present, it is shown that the operator of linear peridynamics diverges, in the limit of vanishing nonlocality, at material interfaces.
Nonlocal interface conditions, whose local limit implies the classical interface conditions of elasticity, are then developed and discussed. 
A peridynamics material interface model is introduced which generalizes
the classical interface model of elasticity. 
The model consists of a new peridynamics operator along with nonlocal interface conditions.
The new  peridynamics interface model   converges
to the classical interface model of linear elasticity.
\end{abstract}



\section{Introduction}
\label{sec_intro}

Peridynamics \cite{Silling39,silling2007peridynamic} is a nonlocal theory for continuum mechanics. 
Material points interact through forces that act over a finite distance with the maximum interaction radius being called the peridynamics {\it horizon}.
Peridynamics is a  generalization to elasticity theory in the sense that 
 peridynamics operators converge to corresponding elasticity operators in the limit of vanishing horizon.
These convergence results have been shown   for different cases; see \cite{emmrich2007well,peridconvg,nonlocal_calc_peridy_2013,tadele_du_2013}. For example,  in a linear isotropic homogeneous medium,
and under certain regularity assumptions on the vector field 
$\bv$, 
it has been shown in \cite{emmrich2007well} that 
\begin{equation}
  \limd\Ls\bv=\Ns\bv \;\;\;\; \mbox{ in } L^\infty(\Omega)^3
\end{equation}
and in \cite{nonlocal_calc_peridy_2013} it has been shown that
\begin{equation}
  \limd \Ldel\bv=\N\bv \;\;\;\; \mbox{ in } H^{-1}(\mathbb{R}^3),
\end{equation}
where $\Omega$ is a bounded domain, $\Ls$ is the bond-based and $\Ldel$ is the state-based linear peridynamics operators, and $\Ns$ and $\N$ are the corresponding linear elasticity operators, respectively (see Section \ref{overview_sec} for the definitions of these operators).

In this work, we study the behavior of linear peridynamics inside heterogeneous media in the limit of vanishing horizon. 
We focus on the linear peridynamics model for solids given in \cite{silling2007peridynamic,lin_Silling}. 
We note that other models for linear peridynamic solids have been proposed; see for example
\cite{aguiar2013constitutive}. 
In Theorem~\ref{thm_convg} and Proposition~\ref{prop1} of this work 
we show that when the vector-field $\bv$ and  the material properties are sufficiently differentiable then  
\begin{equation}
\label{convg1}
 \limd \Ldel\bv=\N\bv \;\;\;\; \mbox{ in } L^p(\Omega)^3, \;\;\; 1\leq p<\infty, 
\end{equation}
and
\begin{equation}
\label{convg2}
 \limd \Ls\bv=\Ns\bv \;\;\;\; \mbox{ in } L^p(\Omega)^3, \;\;\; 1\leq p<\infty,
\end{equation}
where $\Ldel$ and $\Ls$ are the state-based and bond-based  linear peridynamics operator for an isotropic heterogeneous medium and $\N$ and $\Ns$ are the corresponding operators of linear elasticity, respectively.
In addition, we show  that continuity of the material properties is a necessary condition for the convergence of peridynamics to elasticity.
Indeed, if the material properties have jump discontinuities, as for example in multi-phase composites, then 
it is shown in Theorem~\ref{nonconvg_thm} and Lemma~\ref{nonconvg_1} that  the local limits of the 
peridynamic operators do not exist. In particular, we find that for points $\bx$ on the interface,
\begin{equation}
\label{nonconvg1}
  \limd \left(\Ldel\bv\right)(\bx) \mbox{ does not exist},
\end{equation}
and
\begin{equation}
\label{nonconvg2} 
  \limd \left(\Ls\bv\right)(\bx) \mbox{ does not exist}.
 \end{equation}



We consider the classical interface model in linear elasticity inside a two-phase composite. The strong form of the elastic equilibrium problem
is given by the following system of partial differential equations  and interface conditions:
\begin{numcases}{\label{interface_pde0}}
\label{interface_pde0_1}
\Div\sigma(\bx)&$\displaystyle=\bb(\bx),\;\;\; \bx\in\Omega_+$ \\
\label{interface_pde0_2}
\Div\sigma(\bx)&$\displaystyle=\bb(\bx),\;\;\; \bx\in\Omega_-$\\[1.5ex]
\label{interface_pde0_3}
\sigma\bn(\bx^+) &$\displaystyle=\sigma\bn(\bx^-),\;\; \bx\in \Gamma$\\
\label{interface_pde0_4}
\bu(\bx^+) &$\displaystyle=\bu(\bx^-),\;\;\;\; \bx\in \Gamma$,
\end{numcases}
where $\sigma$ is the stress tensor, $\bu$ is the displacement field, $\bb$ is a body force density,
$\bn$ is the unit normal to the interface, and  $\Omega=\Omega_+\cup\Omega_-\cup\Gamma$, 
 with $\Gamma$ being the interface between the two phases $\Omega_+$ and $\Omega_-$. 
Equations \eqref{interface_pde0_3} and \eqref{interface_pde0_4} are the interface jump conditions, 
assuming continuity of the displacement field and traction across the interface.

Developing a  material interface model which  generalizes  the classical interface model of elasticity to the  nonlocal setting is
an open problem in peridynamics.
The fact that  interface conditions are necessary for a classical solution of \eqref{interface_pde0} to exist together with the fact that  the 
peridynamics operator diverges, in the limit of vanishing horizon, at material interfaces
strongly suggests that {\it nonlocal interface conditions} must be imposed in a peridynamics model for heterogeneous media in the presence of material interfaces.
Therefore, a peridynamics interface model which is locally consistent with the interface model of elasticity is required to satisfy the 
following three conditions:
\begin{itemize}
 \item[] \hspace{-.5cm}{\bf C(i)} \hspace{.35cm} Nonlocal interface conditions must be imposed such that the peridynamics  operator converges, in the local limit, to the corresponding elasticity   operator. 
\item[] \hspace{-.5cm}{\bf C(ii)} \hspace{.1cm}  The interface conditions in elasticity are recovered from the local limit of the nonlocal interface conditions in peridynamics.
\item[] \hspace{-.5cm}{\bf C(iii)} \hspace{.1cm}  The nonlocal interface conditions are integral equations that do not include spatial derivatives of the displacement field.
\end{itemize}
We note that  condition C(i) implies that the peridynamics operator is required not to diverge, in the  limit of vanishing horizon, 
at material interfaces. Condition C(iii) requires that the nonlocal interface conditions  be compatible with the peridynamics model. 
Peridynamics is formulated with integral equations and oriented towards modeling discontinuities, and thus peridynamics equations do not include 
spatial derivatives of the displacement field.  

We consider the following peridynamics model, under equilibrium conditions, for heterogeneous media in the presence of material interfaces

\begin{numcases}{\label{natural_interface_sys}}
\label{natural_interface_sys_1}
    \Ldel\bu(\bx) =\bb(\bx),\;\;\; \bx\in\Omega
    \\
\label{natural_interface_sys_2}
\displaystyle \Ldel\bu(\bx)=0,\;\;\;\;\;\;\;\;\, \bx\in \Gamma.
\end{numcases}
Equation \eqref{natural_interface_sys_2} is a nonlocal interface condition. 
By imposing \eqref{natural_interface_sys_2}, and under certain regularity assumptions on the material properties and 
the displacement field, we show that
\[
 \limd \Ldel\bv=\N\bv \;\;\;\; \mbox{ in } L^p(\Omega)^3 
\]
for $1\leq p<\infty$. Thus, the peridynamics interface model \eqref{natural_interface_sys} 
satisfies conditions C(i) and C(iii). However, it is  shown in 
Proposition \ref{prop_wrong_interface_conditions}  that  this model does not satisfy C(ii).  
Therefore, the  interface model  \eqref{natural_interface_sys} 
is not a valid generalization of the local interface model of elasticity.
We note that this result remains true if an inhomogeneity is introduced into \eqref{natural_interface_sys_2}. We also note that, in the nonlocal setting, the material interface remains sharp; however, because of the nonlocality of interactions, nonlocal interface conditions involve points on both sides of the material interface, and not just points on the material interface.

In Section \ref{sec_perid_interface_model}, we propose a solution to the interface problem in peridynamics.
We develop a  peridynamics interface model which is locally consistent with  
elasticity's interface model \eqref{interface_pde0}. Our model is defined by
\begin{numcases}{\label{interface_model_sys}}
\label{interface_model_sys_1}
    \Ldel_*\bu(\bx) =\bb(\bx),\;\;\; \bx\in\Omega
    \\
\label{interface_model_sys_2}
\displaystyle \Ldel_*\bu(\bx)=0,\;\;\;\;\;\;\;\;\, \bx\in \Gamma_\delta,
\end{numcases}
where $\Gamma_\delta$ is an extended interface, which 
is a three-dimensional set of thickness $2\delta$, and the operator $\Ldel_*$ is of the form
\begin{eqnarray}
\label{L*_0_t}
\Ldel_*\bu=\Ldel\bu+1_{\Gamma_\delta}\, \Ldel_{\Gamma_\delta}\bu,
\end{eqnarray}
with $1_{\Gamma_\delta}$ being the indicator function of the set $\Gamma_\delta$. 
The new operator $\Ldel_{\Gamma_\delta}$ acts  on the displacement field but only at points in the extended interface $\Gamma_\delta$. 
The set $\Gamma_\delta$ and the operator $\Ldel_{\Gamma_\delta}$ are explicitly defined in Section 
\ref{sec_perid_interface_model}.
Equation \eqref{interface_model_sys_2} is the peridynamics nonlocal interface condition for our interface model.
By imposing \eqref{interface_model_sys_2},
and under the assumptions that
the material properties are sufficiently  differentiable in $\Omega\setminus\Gamma$ and have jump discontinuities at
the interface $\Gamma$, and that the displacement field $\bu$ is sufficiently  differentiable in $\Omega\setminus\Gamma$ and
continuous across  $\Gamma$, we show in Theorem~\ref{thm_interface_model} that
\[
 \limd \Ldel_*\bv=\N\bv \;\;\;\; \mbox{ in } L^p(\Omega)^3
\]
for $1\leq p<\infty$. Moreover, we show that the local interface condition \eqref{interface_pde0_3} can be recovered from the local limit
of the nonlocal interface condition \ref{interface_model_sys_2}.
Therefore, the peridynamics material interface model \eqref{interface_model_sys} 
satisfies the three conditions C(i)--C(iii) and hence serves as a peridynamics generalization to the classical elasticity interface model.

Here we discuss the mechanical interpretations and implications of the main results in this work.\\
The convergence results given by \eqref{convg1} and \eqref{convg2}, which are introduced in Theorem~\ref{thm_convg} and Proposition~\ref{prop1}, respectively, imply that peridynamics (bond-based or state-based) is a nonlocal generalization of the local continuum theory in the case of isotropic heterogeneous media with smoothly varying material properties. This extends the previous peridynamics convergence results for  homogeneous media  \cite{emmrich2007well,peridconvg,nonlocal_calc_peridy_2013}.

In the case of heterogeneous media with discontinuous material properties, our results given by \eqref{nonconvg1} and \eqref{nonconvg2}, which are introduced in Theorem~\ref{nonconvg_thm} and Lemma~\ref{nonconvg_1}, respectively, imply that the local limit of the peridynamic force is infinite at the material interface. 
This divergence behavior can be explained mathematically through the fact that material interfaces break the inherent symmetry of the peridynamic operators. Mechanically, the divergence of peridynamics is due to the mismatch in the nonlocal tractions on each side of the interface. In fact, the divergence of the local limit of peridynamics at material interfaces is not surprising because in the local interface problem \eqref{interface_pde0} one must impose interface conditions to obtain a well-posed system. 
Therefore, when material interfaces are present,  nonlocal  interface conditions must be imposed in order for peridynamics to converge to a local theory.
The goal of imposing nonlocal interface conditions is to fix the mismatch in  the nonlocal tractions on each side of the interface.
One way to achieve this is by imposing the nonlocal interface condition given by \eqref{natural_interface_sys_2}. Indeed, in Section~\ref{interf_model_sec_1} it is shown that when  \eqref{natural_interface_sys_2} is imposed, then $\Ldel\bu$ converges in the limit as $\delta\rightarrow 0$. However, it is shown in Proposition~\ref{prop_wrong_interface_conditions}
 that the peridynamic system given by \eqref{natural_interface_sys} does not converge to the local elastic interface model given by \eqref{interface_pde0}.
We conclude in Section~\ref{interf_model_sec} that  imposing nonlocal interface conditions alone is not sufficient to achieve a peridynamic interface model that recovers the classical interface model in the local limit.
We therefore propose the peridynamic interface model given by  \eqref{interface_model_sys} which consists of introducing a new  peridynamic operator $\Ldel_*$ together with  imposing a nonlocal interface condition given by \eqref{interface_model_sys_2}. The new operator satisfies $\Ldel_*\bu(\bx)=\Ldel\bu(\bx)$ for points $\bx\in\Omega\setminus\Gamma$. For points $\bx$ on the interface $\Gamma$, the expression $\delta \Ldel_*\bu(\bx)$ represents the jump in the nonlocal traction across the interface. This is justified in Section~\ref{sec_perid_interface_model}  in which it is shown that
\begin{equation}
\label{nonlocal_traction}
\limd \delta\Ldel_* \bu(\bx)= \frac{45}{32} \jump{\sigma}\bn.
\end{equation}
The operator $\Ldel_{\Gamma_\delta}$ in \eqref{L*_0_t},   given explicitly by \eqref{L_Gamma}, which acts on points on the extended interface $\Gamma_\delta$, can be interpreted as the missing term in peridynamics which modifies the jump in the nonlocal traction such  that \eqref{nonlocal_traction}  holds true. It follows from \eqref{nonlocal_traction}, as described in Section~\ref{sec_perid_interface_model}, that the nonlocal interface condition \eqref{interface_model_sys_2} is the nonlocal analogue of the local interface condition \eqref{interface_pde0_3}. 
Theorem~\ref{thm_interface_model} implies that the peridynamic interface model given by \eqref{interface_model_sys} is the nonlocal analogue of the local interface model given by \eqref{interface_pde0}.



This article is organized as follows. Section \ref{overview_sec} provides an overview of linear peridynamics  and linear elasticity inside  
isotropic heterogeneous media. 
The convergence of linear peridynamics operator to linear elasticity operator for the case of heterogeneous media  
is given in Section \ref{convg_sec}. 
The divergence of the peridynamics operator, in the local limit, at material interfaces is addressed in Section \ref{nonconvg_sec}. 
Finally, in Section \ref{interf_model_sec}
nonlocal interface conditions are developed and discussed and our new peridynamics material interface model is introduced and justified.


\section{Overview}
\label{overview_sec}


\subsection{The peridynamics model for solid mechanics}
\label{peridyn}
We consider the state-based peridynamics model introduced in \cite{silling2007peridynamic} for the dynamics of 
deformable solids. 
To simplify the presentation, we provide a direct description of this model 
without adhering to the notation used in \cite{silling2007peridynamic}. 
Following the presentation of peridynamics given in \cite{Alali_Gunzburger1}, 
let $\Omega$ denote a domain in $\bbR^3$, $\bu(\bx,t)$ the displacement vector field, $\rho(\bx)$ the mass density, and $\bb(\bx,t)$ a prescribed body force density. Let $B_\delta(\bx)$ denote the ball centered at $\bx$ having radius $\delta$; here, $\delta$ denotes the peridynamics horizon. 
Then the linear peridynamics equation of motion for an isotropic heterogeneous medium is given by
\begin{equation}
\label{pd1}
 \rho(\bx) \ddot{\bu}(\bx,t) = (\Ldel \bu)(\bx) + \bb(\bx,t),\;\;\;\;\; \bx\in\Omega,
\end{equation}
where
\begin{equation}
\label{Ldel0}
\Ldel=\Ldel_s+\Ldel_d,
\end{equation}
and, for a vector field $\bv$, the operators $\Ldel_s$ and $\Ldel_d$ are given by
\small
\begin{eqnarray}
\label{Ldel_s_1}
(\Ldel_s \bv)(\bx)&=& \int_{B_\delta(\bx)}
  \frac{15}{m}\big(\mu(\bx)+\mu(\by)\big) w(|\by-\bx|)\frac{(\by-\bx)
\otimes(\by-\bx)}{|\by-\bx|^2}\big( \bv(\by)-\bv(\bx)\big)
  \,d\by,\\
  \nonumber\\
  \nonumber
  (\Ldel_d \bv)(\bx)&=& \int_{B_\delta(\bx)}\int_{B_\delta(\bx)}
  \frac{9}{m^2}
 \Big(\lambda(\bx)-\mu(\bx)\Big) w(|\by-\bx|)w(|\bz-\bx|)(\by-\bx)
\otimes(\bz-\bx)\big( \bv(\bz)-\bv(\bx)\big)
 \,d\bz d\by\\
 \nonumber\\
 \nonumber
  &&+ \int_{B_\delta(\bx)}\int_{B_\delta(\by)}
  \frac{9}{m^2}
 \Big(\lambda(\by)-\mu(\by)\Big) w(|\by-\bx|)w(|\bz-\by|)
 (\by-\bx)
\otimes(\bz-\by)\big( \bv(\bz)-\bv(\by)\big)
 \,d\bz d\by.\\
 \label{Ldel_d_1}
\end{eqnarray}
\normalsize
Here $\lambda$ and $\mu$ are Lam\'{e} parameters, with $\mu$ denoting the shear modulus, $w$ is a weighting function, and $m$ denotes a scalar weight given by $m=\int_\Omega w(|\by-\bx|) |\by-\bx|^2 d\by$. Since $w$ is a radial function in (\ref{Ldel_s_1}), (\ref{Ldel_d_1}), the material is isotropic, and $w$ can be taken to be of the form (see, for example, \cite{Silling39})
\begin{equation}
 w(|\xi|) = \left\{\begin{aligned}
                    \frac{1}{|\bxi|^r}\;, \qquad& \mbox{if $|\bxi|<\delta$} \\
                    0\;,                 \qquad& \text{otherwise}.
                   \end{aligned}
\right.
\end{equation}
In this case 
\begin{equation}
 m = \int_{B_{\delta}(0)} |\bxi|^{2-r} d\bxi = 4\pi\frac{\delta^{5-r}}{5-r}.
\end{equation}
Note that when $r<5$, $m$ is finite. To simplify the presentation, and without loss of generality, we assume that $r=2$;
consequently, $m = \frac{4}{3}\pi\delta^3 = |B_{\delta}|$ and
\begin{eqnarray}
\label{Ldel_s_2}
(\Ldel_s \bv)(\bx)&=& \frac{15}{|B_\delta|}\int_{B_\delta(\bx)}
  \big(\mu(\bx)+\mu(\by)\big) \frac{(\by-\bx)
\otimes(\by-\bx)}{|\by-\bx|^4}\big( \bv(\by)-\bv(\bx)\big)
  \,d\by,\\
  \nonumber\\
  \nonumber
  (\Ldel_d \bv)(\bx)&=& \frac{9}{|B_\delta|^2}\int_{B_\delta(\bx)}\int_{B_\delta(\bx)}
 \Big(\lambda(\bx)-\mu(\bx)\Big) 
 \frac{\by-\bx}{|\by-\bx|^2}
\otimes\frac{\bz-\bx}{|\bz-\bx|^2}\big( \bv(\bz)-\bv(\bx)\big)
 \,d\bz d\by\\
 \nonumber\\
 \nonumber
  & & + \frac{9}{|B_\delta|^2}\int_{B_\delta(\bx)}\int_{B_\delta(\by)}
 \Big(\lambda(\by)-\mu(\by)\Big) \frac{\by-\bx}{|\by-\bx|^2}
\otimes\frac{\bz-\by}{|\bz-\by|^2}\big( \bv(\bz)-\bv(\by)\big)
 \,d\bz d\by.\\
 \label{Ldel_d_2}
\end{eqnarray}
Due to symmetry we have the following identity
\begin{equation}
\label{sym}
\int_{B_\delta(\bp)}\frac{\bq-\bp}{|\bq-\bp|^2}\,d\bp=0.
\end{equation}
For points $\bx\in\Omega$ with a distance of at least $2\delta$ from the boundary $\partial \Omega$, 
the operator  $\Ldel_d$ in  \eqref{Ldel_d_2} reduces to
\begin{equation}
\label{Ldel_d_3}
(\Ldel_d \bv)(\bx)=\frac{9}{|B_\delta|^2}\int_{B_\delta(\bx)}\int_{B_\delta(\by)}
 \Big(\lambda(\by)-\mu(\by)\Big) \frac{\by-\bx}{|\by-\bx|^2}
\otimes\frac{\bz-\by}{|\bz-\by|^2}\,\bv(\bz)
 \,d\bz d\by,
\end{equation}
where we have applied (\ref{sym}).
Throughout this article, we will use the notation $A\colon B$ to denote the inner product of the same-order tensors $A$ and $B$. For example, if $A$ and $B$ are third-order tensors then
\[
A\colon B = \sum_{i,j,k} A_{i j k} B_{i j k}.
\]


\subsection{Linear elasticity}
\label{sec_elasticity}
In linear elasticity, the stress tensor for an isotropic heterogeneous medium is given by
\begin{equation}
\label{sigma}
\sigma(\bx)=\lambda(\bx) \Div\bu(\bx)\, I+\mu(\bx)(\nabla\bu(\bx)+\nabla\bu(\bx)^T),
\end{equation}
where $\bu$ is the displacement field, $I$ is the identity tensor, and $\lambda$ and $\mu$ are Lam\'{e} parameters. The equation of motion in this case is given by
\begin{equation}
\label{lin_elast}
 \rho(\bx) \ddot{\bu}(\bx,t) = (\N \bu)(\bx) + \bb(\bx,t),\;\;\;\;\; \bx\in\Omega,
\end{equation}
where $\N$ is the Navier operator of linear elasticity which is given by 
\begin{eqnarray}
\label{N}
\nonumber
\N\bu &=& \Div\sigma\\
&=& \nabla(\lambda \Div\bu) + \Div(\mu (\nabla\bu+\nabla\bu^T)).
\end{eqnarray}
We decompose the operator of linear elasticity as
\begin{equation}
\label{N_decomp}
\N=\Ns+\Nd,
\end{equation}
where the operators $\Ns$ and $\Nd$ are defined by 
\begin{eqnarray}
\label{Ns}
\Ns\bv &=& \nabla(\mu \Div\bv) + \Div(\mu (\nabla\bv+\nabla\bv^T)),\\
\Nd\bv &=& \nabla((\lambda-\mu) \Div\bv),
\end{eqnarray}
for sufficiently regular vector field $\bv$. 

We note that the above decomposition of $\N$ is not a standard one; however, this decomposition will be useful for studying the relationship between the nonlocal operator of peridynamics $\Ldel$, defined in Section \ref{peridyn}, 
and the local operator of elasticity $\N$; see Section \ref{convg_sec}. 
\begin{remark}
It is easy to see that $\N=\Ns$ for materials in which $\lambda=\mu$ or, equivalently, materials with Poisson ratio $\nu=\frac{1}{4}$.
\end{remark}


\section{Convergence of Linear Peridynamics to Linear Elasticity Inside Heterogeneous Media}
\label{convg_sec}

In this section we show that in a heterogeneous medium and under certain regularity assumptions on the material properties and the 
vector field $\bv$, the linear peridynamics operator $\Ldel$ converges to the linear elasticity operator $\N$ in 
the limit of vanishing horizon. This is given by Theorem \ref{thm_convg} in the last part of this section.  

We start by defining an operator $\Lz$, which is independent of material properties,
\begin{eqnarray}
\label{Lz}
\nonumber
(\Lz \bv)(\bx)&=& \frac{30}{|B_\delta|}\int_{B_\delta(\bx)}
  \frac{(\by-\bx)
\otimes(\by-\bx)}{|\by-\bx|^4}\big( \bv(\by)-\bv(\bx)\big)
  \,d\by\\
  &=& \frac{30}{|B_\delta|}\int_{B_\delta(0)}
   \frac{\bz
\otimes\bz}{|\bz|^4}\,\big( \bv(\bx+\bz)-\bv(\bx)\big)
  \,d\bz,
\end{eqnarray}
where   $\bv$ is a vector field and $\bx\in \mathbb{R}^3$. 
We note $\Lz$ is a bounded linear operator on $L^p(\Omega)^3$ for $1\leq p\leq \infty$; see for example \cite{ALperid1}. \\
\begin{lem}
\label{lemma1}
If $\bv\in C^3(\Omega)^3$ then
\begin{equation}
\label{lim_Lz}
\limd \Lz\bv=2 \nabla(\Div\bv) +\Delta\bv,\;\;\;\; \mbox{ in } L^p(\Omega)^3\end{equation}
for $1\leq p<\infty$.
\end{lem}
\begin{proof}
The Taylor expansion of $\bv$ about $\bz=\bx$ is given by
\begin{equation}
\label{taylor1}
\bv(\bx+\bz)=\bv(\bx)+\nabla\bv(\bx) \bz+\frac{1}{2} \nabla\nabla\bv(\bx)\, (\bz\otimes\bz)+\br(\bv;\bx,\bz),
\end{equation}
where
\begin{equation}
\label{remainder1}
\br(\bv;\bx,\bz)=\frac{1}{3!} \nabla\nabla\nabla\bv(\bx+ t \bz)\, (\bz\otimes\bz\otimes\bz)
\end{equation}
for some $t\in(0,1)$. By inserting \eqref{taylor1} in \eqref{Lz}, expanding the integral, and then rearranging the tensor products, we obtain
\begin{eqnarray}
\label{Lz2}
\nonumber
(\Lz \bv)(\bx)&=& \frac{30}{|B_\delta|}\int_{B_\delta(0)}
   \frac{\bz\otimes\bz\otimes \bz}{|\bz|^4}\,d\bz\, \nabla\bv(\bx)+
  \frac{30}{|B_\delta|}\int_{B_\delta(0)}
   \frac{\bz\otimes\bz\otimes \bz\otimes\bz}{|\bz|^4}\,d\bz\, \frac{1}{2} \nabla\nabla\bv(\bx)\\
   & & + \frac{30}{|B_\delta|}\int_{B_\delta(0)}
   \frac{\bz\otimes\bz}{|\bz|^4}\,\br(\bv;\bx,\bz)\,d\bz.
\end{eqnarray}
We note that, due to symmetry, the integral
\begin{equation}
\label{zzz0}
\int_{B_\delta(0)}
   \frac{\bz\otimes\bz\otimes \bz}{|\bz|^4}\,d\bz=0,
\end{equation}
with the obvious notation that $0$ in the right hand side of \eqref{zzz0} denotes the third-order zero tensor. Thus the first term in \eqref{Lz2} vanishes. 
We note that the third term in \eqref{Lz2} vanishes in the limit as $\delta\rightarrow 0$ because
\begin{eqnarray}
\label{O_delta}
\left|\frac{30}{|B_\delta|}\int_{B_\delta(0)}
   \frac{\bz\otimes\bz}{|\bz|^4}\,\br(\bv;\bx,\bz)\,d\bz\; \right| &\leq&
   \frac{M}{\delta^3} \int_{B_\delta(0)} |\bz|\,d\,\bz=\BigO{\delta},
\end{eqnarray}
for some $M>0$.
For the second term in \eqref{Lz2}, a straightforward calculation, using spherical coordinates, shows that the following fourth-order tensor
satisfies
\begin{equation}
\label{zzzz1}
\frac{30}{|B_\delta|}\int_{B_\delta(0)}
   \frac{z_i z_j z_k z_l}{|\bz|^4}\,d\bz=
\left\{
  \begin{array}{ll}
    6, & \mbox{ if } i=j=k=l,\\
    \\
    2, & \mbox{ if } (i=j,k=l, \mbox{ and } i\neq k) \\
     & \;\; \mbox{ or } (i=k,j=l, \mbox{ and } i\neq j)\\
      & \;\; \mbox{ or } (i=l,j=k, \mbox{ and } i\neq j),\\
    \\
    0, & \mbox{otherwise}.
  \end{array}
\right.
\end{equation}
Using \eqref{zzzz1} the $i$-th component of the second term in \eqref{Lz2} becomes
\begin{eqnarray}
\label{T2}
\nonumber
\sum_{j,k,l} \frac{30}{|B_\delta|}\int_{B_\delta(0)}
   \frac{z_i z_j z_k z_l}{|\bz|^4}\,d\bz\, \frac{1}{2} \frac{\partial^2 v_j}{\partial x_l\partial x_k} 
   &=& \frac{1}{2} \left( 6 \frac{\partial^2 v_i}{\partial x_i^2}+
   2 \sum_{k\neq i}\frac{\partial^2 v_i}{\partial x_k^2}+
   4 \sum_{j\neq i}\frac{\partial^2 v_j}{\partial x_i x_j}
   \right)\\
   \nonumber
   &=& \sum_{k} \frac{\partial^2 v_i}{\partial x_k^2} 
   + 2 \sum_{j} \frac{\partial^2 v_j}{\partial x_i x_j}  \\
   &=& \Delta v_i + 2 \left(\nabla(\Div \bv)\right)_i.
\end{eqnarray}
By combining \eqref{Lz2} with \eqref{zzz0}, \eqref{O_delta}, and \eqref{T2}, we conclude that
\begin{equation}
\label{lim_Lz_pt}
\limd (\Lz \bv)(\bx)= 2 \nabla(\Div\bv(\bx)) +\Delta\bv(\bx)
\end{equation}
for all $\bx$ in $\mathbb{R}^3$. Equation \eqref{lim_Lz} follows from the point-wise convergence result \eqref{lim_Lz_pt} and Lebesgue's dominated convergence theorem, completing the proof.
\end{proof}

The operator $\Lz: L^p(\Omega)^3\rightarrow L^p(\Omega)^3$ can also be defined to operate on scalar-fields
\[
(\Lz f)(\bx):=\frac{30}{|B_\delta|}\int_{B_\delta(0)}
   \frac{\bz
\otimes\bz}{|\bz|^4}\,\big( f(\bx+\bz)-f(\bx)\big)
  \,d\bz,
\]
in which case $f\in L^p(\Omega)\mapsto \Lz f\in L^p(\Omega)^{3\times 3}$. The convergence result in this case is given by the following lemma, whose proof is similar to that of Lemma \ref{lemma1}.
\begin{lem}
\label{lemma2}
If $f\in C^3(\Omega)$ then
\begin{equation}
\label{lim_Lz_scalar}
\limd \Lz f
=2 \nabla\nabla f +\Delta f \;I,\;\;\;\; \mbox{ in } L^p(\Omega)^{3\times 3}\end{equation}
for $1\leq p<\infty$.
\end{lem}
We use Lemma~\ref{lemma1} and Lemma~\ref{lemma2} to show the following convergence result for the operator $\Ls$ defined in \eqref{Ldel_s_2}.
\begin{prop}
\label{prop1}
Assume that the vector field $\bv$ is in $C^3(\Omega)^3$ and the shear modulus $\mu$ is in $C^3(\Omega)$. Then as $\delta \rightarrow 0$,
\begin{equation}
\label{lim_Ls}
\Ls \bv\longrightarrow \Ns\bv,\;\;\;\; \mbox{ in } L^p(\Omega)^3
\end{equation}
for $1\leq p<\infty$.
\end{prop}
\begin{proof}
The operator $\Ls$ in \eqref{Ldel_s_2}, after the change of variables $\bz=\by-\bx$, becomes
\begin{eqnarray}
\label{Ls3}
(\Ls\bv)(\bx)=\frac{15}{|B_\delta|}\int_{B_\delta(0)}
  \left(\mu(\bx)+\mu(\bx+\bz)\right) \frac{\bz
\otimes\bz}{|\bz|^4}\,\big( \bv(\bx+\bz)-\bv(\bx)\big)\,d\bz.
\end{eqnarray}
We decompose the operator $\Ls$ as $\Ls=\Ldel_{s1}+\Ldel_{s2}$, where
\begin{eqnarray}
\label{Ls1}
(\Ldel_{s1}\bv)(\bx)&=&\frac{1}{2}\mu(\bx)\;\frac{30}{|B_\delta|}\int_{B_\delta(0)}
   \frac{\bz
\otimes\bz}{|\bz|^4}\,\big( \bv(\bx+\bz)-\bv(\bx)\big)\,d\bz,\\
\label{Ls2}
(\Ldel_{s2}\bv)(\bx)&=&\frac{1}{2}\;\frac{30}{|B_\delta|}\int_{B_\delta(0)}
 \mu(\bx+\bz)\;  \frac{\bz
\otimes\bz}{|\bz|^4}\,\big( \bv(\bx+\bz)-\bv(\bx)\big)\,d\bz,
\end{eqnarray}
Using Lemma~\ref{lemma1} we find that, as $\delta\rightarrow 0$
\begin{equation}
\label{Ls1_2}
\Ldel_{s1}\bv\longrightarrow \frac{1}{2}\, \mu \left( \frac{}{}
2 \nabla(\Div\bv) +\Delta\bv\right),\;\;\;\; \mbox{ in } L^p(\Omega)^3.
\end{equation}
The integral in \eqref{Ls2} can be written as
\begin{eqnarray}
\label{Ls2_2}
\nonumber
(\Ldel_{s2}\bv)(\bx)&=&\frac{1}{2}\;\frac{30}{|B_\delta|}\int_{B_\delta(0)}
  \frac{\bz
\otimes\bz}{|\bz|^4}\,\big(\mu(\bx+\bz) \bv(\bx+\bz)-\mu(\bx) \bv(\bx)\big)\,d\bz\\
&& - \frac{1}{2}\;\frac{30}{|B_\delta|} \bv(\bx) \int_{B_\delta(0)}
  \frac{\bz
\otimes\bz}{|\bz|^4}\,\big(\mu(\bx+\bz)-\mu(\bx) \big)\,d\bz.
\end{eqnarray}
By applying Lemma~\ref{lemma1} on the first term of the right hand side of \eqref{Ls2_2} and Lemma~\ref{lemma2} on the second term  we find that, as $\delta\rightarrow 0$
\begin{equation}
\label{Ls2_convg}
\Ldel_{s2}\bv\longrightarrow \frac{1}{2} \left( \frac{}{}
2 \nabla(\Div(\mu\bv)) +\Delta(\mu\bv)\right) - 
\frac{1}{2} \left( \frac{}{}
\left(2 \nabla\nabla \mu  +\Delta\mu) \,I\right)\bv\right),
\end{equation}
in $L^p(\Omega)^3$. From \eqref{Ls1_2} and \eqref{Ls2_convg}, we obtain the following convergence in $L^p(\Omega)^3$
\begin{eqnarray}
\label{Ls_convg}
\limd\Ldel_{s}\bv &=&\frac{1}{2}\left( \frac{}{}
2\mu\nabla(\Div\bv)+\mu\Delta\bv+
2 \nabla(\Div(\mu\bv)) +\Delta(\mu\bv) - 
2 (\nabla\nabla \mu)\bv  -(\Delta\mu)\bv
\right).
\end{eqnarray}
Expanding $\Delta(\mu\bv)$ and $\nabla(\Div(\mu\bv))$ in the right hand side of \eqref{Ls_convg}, using the identities
\begin{eqnarray}
\label{calc_identities}
\nonumber
\Delta(\mu\bv) &=& \mu \Delta\bv+2\nabla\bv\nabla\mu+(\Delta\mu)\;\bv,\\
\nonumber
\nabla(\Div(\mu\bv)) &=& (\nabla\nabla\mu)^T\bv+(\nabla\bv)^T \nabla\mu+\mu\nabla(\Div\bv)+(\Div\bv)\nabla\mu,
\end{eqnarray}
and then simplifying, one finds that
\begin{eqnarray}
\label{Ls_convg_2}
\nonumber
&&\frac{1}{2}\left( \frac{}{}
2\mu\nabla(\Div\bv)+\mu\Delta\bv+
2 \nabla(\Div(\mu\bv)) +\Delta(\mu\bv) - 
2 (\nabla\nabla \mu)\bv  -(\Delta\mu)\bv
\right)\\
\nonumber
&&\;\;\;\;\;\;\;= (\Div\bv)\nabla\mu+2\mu\nabla(\Div\bv)+\mu\Delta\bv+\nabla\bv\nabla\mu+(\nabla\bv)^T\nabla\mu\\
&&\;\;\;\;\;\;\;= \nabla\left(\mu\Div\bv\right) + 
\Div\left(\mu(\left(\nabla\bv+(\nabla\bv)^T\right)\right).
\end{eqnarray}
Finally, equation \eqref{lim_Ls} follows from \eqref{Ns}, \eqref{Ls_convg},
and \eqref{Ls_convg_2}.
\end{proof}
In the next result we consider the convergence of the operator $\Ld$ defined in \eqref{Ldel_d_3}. 
\begin{prop}
\label{prop2}
Assume that the vector field $\bv$ is in $C^3(\Omega)^3$ and that the material properties
$\mu$ and $\lambda$ are in $C^2(\Omega)$. Then as $\delta \rightarrow 0$,
\begin{equation}
\label{lim_Ld}
\Ld \bv\longrightarrow \Nd\bv,\;\;\;\; \mbox{ in } L^p(\Omega)^3
\end{equation}
for $1\leq p<\infty$.
\end{prop}
\begin{proof}
Let $\bx$ be a point in the interior of $\Omega$ and  $c=\lambda-\mu$. Then by changing variables ($\bw=\bz-\by$ then $\bh=\by-\bx$) in \eqref{Ldel_d_3}, $\Ld\bv$ can be written as
\begin{eqnarray}
\label{Ld}
\nonumber
(\Ld \bv)(\bx)&=&\frac{9}{|B_\delta|^2}\int_{B_\delta(0)}\int_{B_\delta(0)}
 c(\bx+\bh) \;\frac{\bh\otimes\bw}{|\bh|^2 |\bw|^2}
\,\bv(\bx+\bh+\bw)\,d\bw\; d\bh\\
&=&\frac{9}{|B_\delta|^2}\int_{B_\delta(0)}
 c(\bx+\bh) \;\frac{\bh}{|\bh|^2}\int_{B_\delta(0)}\frac{\bw}{|\bw|^2} \cdot
\,\bv(\bx+\bh+\bw)\,d\bw \,d\bh.
 \end{eqnarray}
The Taylor expansion of $\bv$ about $\bw=\bx+\bh$ is given by
\begin{equation}
\label{Ld_taylor1}
\bv(\bx+\bh+\bw)=\bv(\bx+\bh)+\nabla\bv(\bx+\bh) \bw+\frac{1}{2} \nabla\nabla\bv(\bx+\bh)\, (\bw\otimes\bw)+\br_1(\bv;\bx+\bh,\bw),
\end{equation}
where
\begin{equation}
\label{Ld_remainder1}
\br_1(\bv;\bx+\bh,\bw)=\frac{1}{3!} \nabla\nabla\nabla\bv(\bxi)\, (\bw\otimes\bw\otimes\bw)
\end{equation}
for some $\bxi$ on the line segment joining $\bx+\bh$ and $\bw$. By inserting \eqref{Ld_taylor1} in the inner integral of \eqref{Ld}, expanding the integral, and then rearranging the tensor products, we find
\begin{eqnarray}
\label{Ld_inner_1}
\nonumber
\int_{B_\delta(0)}\frac{\bw}{|\bw|^2}\cdot\,\bv(\bx+\bh+\bw)\,d\bw &=& 
\int_{B_\delta(0)}\frac{\bw}{|\bw|^2}\,d\bw\cdot \bv(\bx+\bh)+
\int_{B_\delta(0)}\frac{\bw\otimes\bw}{|\bw|^2}\,d\bw\colon \nabla\bv(\bx+\bh)\\
\nonumber
&&\hspace{-2cm} +
\int_{B_\delta(0)}\frac{\bw\otimes\bw\otimes\bw}{|\bw|^2}\,d\bw\colon \frac{1}{2}\nabla\nabla\bv(\bx+\bh)
+ \int_{B_\delta(0)}\frac{\bw}{|\bw|^2}\cdot\,\br(\bv;\bx+\bh,\bw)\,d\bw.\\
\end{eqnarray}
We note that due to symmetry, the integrals in the first and third terms of the right hand side of \eqref{Ld_inner_1} are identical to zero. 
A straightforward calculation shows that
\begin{equation}
\label{identity_identity}
\int_{B_\delta(0)}\frac{\bw\otimes\bw}{|\bw|^2}\,d\bw = \frac{|B_\delta|}{3} \,I,
\end{equation}
and thus \eqref{Ld_inner_1} is equivalent to
\begin{eqnarray}
\label{Ld_inner_2}
\nonumber
\int_{B_\delta(0)}\frac{\bw}{|\bw|^2}\cdot\,\bv(\bx+\bh+\bw)\,d\bw &=& 
\frac{|B_\delta|}{3} \; \Div\bv(\bx+\bh)
+ \int_{B_\delta(0)}\frac{\bw}{|\bw|^2}\cdot\,\br(\bv;\bx+\bh,\bw)\,d\bw.\\
\end{eqnarray}
Substituting \eqref{Ld_inner_2} in \eqref{Ld} one finds that
\begin{eqnarray}
\label{Ld_taylor2}
\nonumber
(\Ld \bv)(\bx)&=&
\frac{3}{|B_\delta|}\int_{B_\delta(0)}c(\bx+\bh)\Div\bv(\bx+\bh)\frac{\bh}{|\bh|^2}\,d\bh\\
&& +\frac{9}{|B_\delta|^2}\int_{B_\delta(0)}c(\bx+\bh)\;\frac{\bh}{|\bh|^2}
\int_{B_\delta(0)}\frac{\bw}{|\bw|^2} \cdot
\,\br(\bv;\bx+\bh,\bw)\,d\bw \,d\bh.
\end{eqnarray}
Using \eqref{Ld_remainder1} we obtain that for some $K>0$,
\begin{eqnarray}
\label{Ld_estimate}
\nonumber
\left| 
\frac{9}{|B_\delta|^2}\int_{B_\delta(0)}c(\bx+\bh)\;\frac{\bh}{|\bh|^2}
\int_{B_\delta(0)}\frac{\bw}{|\bw|^2} \cdot
\,\br(\bv;\bx+\bh,\bw)\,d\bw \,d\bh
\right| &\leq& K \frac{9}{|B_\delta|^2} \int_{B_\delta(0)}\frac{1}{|\bh|}\,d\bh \int_{B_\delta(0)}|\bw|^2 \,d\bw\\
&=& \BigO{\delta}, 
\end{eqnarray}
where we used the facts that $\int_{B_\delta(0)}\frac{1}{|\bh|}\,d\bh=\BigO{\delta^2}$ and 
$\int_{B_\delta(0)}|\bw|^2 \,d\bw =\BigO{\delta^5}$. Therefore, the second term in the right hand side of \eqref{Ld_taylor2} vanishes in the limit as $\delta\rightarrow 0$.
For the first term in the right hand side of \eqref{Ld_taylor2}, we first expand $c\Div\bv$ in a Taylor series about $\bh=\bx$
\begin{equation}
\label{Ld_taylor3}
\left(c \Div\bv\right)(\bx+\bh)=\left(c \Div\bv\right)(\bx)+\nabla\left(c \Div\bv\right)(\bx)\cdot \bh+\br_2(\bv;\bx,\bh),
\end{equation}
where
\begin{equation}
\label{Ld_remainder2}
\br_2(\bv;\bx,\bh)=\frac{1}{2} \nabla\nabla\left(c \Div\bv\right)(\bxi)\colon (\bh\otimes\bh)
\end{equation}
for some $\bxi$ on the line segment joining $\bx$ and $\bh$. Then by substituting \eqref{Ld_taylor3} in the first term in the right hand side of \eqref{Ld_taylor2}, we find
\begin{eqnarray}
\label{Ld_taylor2_first}
\nonumber
\frac{3}{B_\delta}\int_{B_\delta(0)}c(\bx+\bh) \Div\bv(\bx+\bh)\; \frac{\bh}{|\bh|^2}\,d\bh &=& 
\frac{3}{B_\delta}\int_{B_\delta(0)} \frac{\bh}{|\bh|^2}\,d\bh\,\, c(\bx) \Div\bv(\bx)\\
\nonumber
&&
\nonumber
+
\frac{3}{B_\delta}\int_{B_\delta(0)} \frac{\bh\otimes\bh}{|\bh|^2}\,d\bh\, \nabla(c\Div\bv)(\bx)
+\frac{3}{B_\delta}\int_{B_\delta(0)} \br_2\frac{\bh}{|\bh|^2}\,d\bh \\
\nonumber
\\
&=& \nabla(c\Div\bv)(\bx) + \BigO{\delta}.
\end{eqnarray}
We note that, in order to obtain \eqref{Ld_taylor2_first} we used \eqref{identity_identity}, the identity
\begin{equation}
\label{zero_integral_1}
\int_{B_\delta(0)}\frac{\bh}{|\bh|^2}\,d\bh = 0,
\end{equation}
and the estimate
\begin{eqnarray}
\label{Ld_estimate2}
\nonumber
\left| 
\frac{3}{B_\delta}\int_{B_\delta(0)}\br_2(\bv;\bx,\bh)\;\frac{\bh}{|\bh|^2} \,d\bh\right| 
&\leq& M \frac{3}{B_\delta}\int_{B_\delta(0)} |\bh|\,d\bh \\
&=& \BigO{\delta} 
\end{eqnarray}
for some $M>0$.
By substituting \eqref{Ld_taylor2_first} in \eqref{Ld_taylor2} and using \eqref{Ld_estimate}, one finds
\begin{equation}
\label{Ld_lim2}
\limd (\Ld \bv)(\bx) = \nabla(c\Div\bv)(\bx).
\end{equation}
The result follows from \eqref{Ld_lim2} and  Lebesgue's dominated convergence theorem. 
\end{proof}
We conclude this section by the following result which follows from combining Propositions \ref{prop1} and \ref{prop2}.
\begin{thm}
\label{thm_convg}
Assume that the vector field $\bv$ is in $C^3(\Omega)^3$ and that the material properties
$\mu$ and $\lambda$ are in $C^3(\Omega)$. Then 
\begin{equation}
\label{lim_Ldel_2}
\limd(\Ldel \bv)(\bx)= (\N\bv)(\bx),\;\;\;\; \bx\in\mathring{\Omega}.
\end{equation}
Moreover, as $\delta \rightarrow 0$,
\begin{equation}
\label{lim_Ldel}
\Ldel \bv\longrightarrow \N\bv\;\;\;\; \mbox{ in } L^p(\Omega)^3,
\end{equation}
for $1\leq p<\infty$.
\end{thm}
\begin{remark}
The regularity assumptions on the vector field $\bv$ and the material properties $\mu$ and $\lambda$ in 
Theorem~\ref{thm_convg} as well as in the other results in this section can be relaxed. However, in Section~\ref{nonconvg_sec}, we 
show that the material properties must at least be continuous for the convergence of peridynamics to elasticity results to hold.  
\end{remark}

\section{Non-Convergence of Peridynamics at Interfaces}
\label{nonconvg_sec}
In this section we show that continuity of the material properties is a necessary condition for the convergence of 
linear peridynamics to linear elasticity as described in Theorem \ref{thm_convg}. 

Let $\Gamma$ be an interface separating different phases inside a heterogeneous medium occupying the region  $\Omega$, 
as illustrated in Figure \ref{fig_gamma}. 
We assume that the surface $\Gamma$ is $C^1$. In this case, the material properties $\lambda$ and $\mu$ have jump 
discontinuities at the interface. To simplify the presentation, we assume that the medium is a two-phase composite with 
$\Omega=\Omega_+\cup\Omega_-\cup\Gamma$, where $\Omega_+$ and $\Omega_-$ are two open disjoint sets, and the material 
properties are piecewise constants, which are given by
\begin{equation}
\label{lambda}
\lambda(\bx)=
\left\{
  \begin{array}{ll}
    \lambda_+,\; \bx\in\Omega_+\cup\Gamma\\
    \lambda_-,\; \bx\in\Omega_-
  \end{array}
\right.,\;\;\;\;\;\;\;\;\;\;\;
\mu(\bx)=
\left\{
  \begin{array}{ll}
    \mu_+,\; \bx\in\Omega_+\cup\Gamma\\
    \mu_-,\; \bx\in\Omega_-
  \end{array}
\right.
\end{equation}
In the remaining part of this article, we will use the following notations. Given a point $\bx_0\in\Gamma$, let $\bn(\bx_0)$ be the unit
normal to the interface at  $\bx_0$. We suppose that $\bn$ is directed outward from the $-$side
of the interface, pointing toward the $+$ side, as illustrated in Figure \ref{fig_gamma}. For a scalar , vector, or, tensor field $F$, we define
\begin{eqnarray*}
F(\bx_0^+)&:=& \lim_{\by\rightarrow \bx_0, \by\in\Omega_+} F(\by), \\
F(\bx_0^-)&:=& \lim_{\by\rightarrow \bx_0, \by\in\Omega_-} F(\by).
\end{eqnarray*}
In addition, we define
\begin{eqnarray*}
\label{Bpm}
\Bp{\bx_0}&:=& B_\delta(\bx_0)\cap\Omega_+\cap\Gamma,\\
\Bm{\bx_0}&:=& B_\delta(\bx_0)\cap\Omega_-.
\end{eqnarray*}
Note that the sets $\Bp{\bx_0}$ and $\Bm{\bx_0}$ depend on the normal $\bn$. Furthermore, we use the following notation to denote
the jump in $F$ across the interface
\[
 \left[F\right]^{+}_{-}:=F(\bx^+)-F(\bx^-),\;\;\; \bx\in\Gamma
\]

\begin{figure}[t]
\centering
    \includegraphics[width=.3\textwidth]{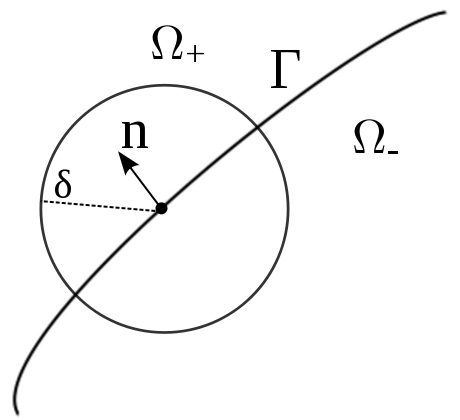}
\caption{\label{fig_gamma} The interface $\Gamma$ separates the two phases $\Omega_+$ and $\Omega_-$.}   
\end{figure}

The behavior of the operator $\Ls$ at material interfaces, in the local limit, is described by the following result.
\begin{lem}
\label{nonconvg_1}
 Assume that the shear modulus $\mu$ is given by \eqref{lambda} and that the vector field $\bv$ 
is continuous on $\Omega$ and smooth on $\Omega\setminus\Gamma$.
Then for $\bx\in\Gamma$,
\[
 \limd \left(\Ls\bv\right)(\bx) \mbox{ does not exist}.
\]
Moreover, the sequence $\left(\Ls\bv\right)_\delta$ is unbounded in $L^p(\Omega)$, with $1\leq p<\infty$.
\end{lem}
\begin{remark}
This result holds in the more general case when $\mu$ is differentiable on $\Omega\setminus\Gamma$ and has a jump discontinuity
across the interface $\Gamma$ rather than just piecewise constant.
\end{remark}
\begin{proof}
 Let $\bx$ be a point on the interface $\Gamma$ away from $\partial\Omega$ by a distance of at least $\delta$.
Then $(\Ls\bv)(\bx)$ in \eqref{Ldel_s_2}, after a change of variables and using the fact that $B_\delta(\bx)=\Bp{\bx}\cup\Bm{\bx}$, 
can be written as
\begin{eqnarray}
\label{Ldel_s_3}
\nonumber
(\Ldel_s \bv)(\bx)&=& \frac{15}{|B_\delta|}\int_{\Bp{\bzz}}
  \big(\mu_+ +\mu_+\big) \frac{\bz
\otimes\bz}{|\bz|^4}\big( \bv(\bx+\bz)-\bv(\bx)\big)
  \,d\bz\\
&& + \frac{15}{|B_\delta|}\int_{\Bm{\bzz}}
  \big(\mu_+ +\mu_-\big) \frac{\bz
\otimes\bz}{|\bz|^4}\big( \bv(\bx+\bz)-\bv(\bx)\big)
  \,d\bz,
\end{eqnarray}
where $\Bp{\bzz}=\Bp{\bx}-\bx$ and $\Bm{\bzz}=\Bm{\bx}-\bx$. Note that in \eqref{Ldel_s_3} we have used the facts that, for $\bx\in\Gamma$,
 $\mu(\bx)=\mu_+$, $\mu(\bx+\bz)=\mu_+$ for $\bz\in\Bp{\bzz}$, and $\mu(\bx+\bz)=\mu_-$ for $\bz\in\Bm{\bzz}$.
Since $\bv$ is smooth on each side of $\Gamma$, then for $\bz$ in the $+$side of $\Gamma$ (i.e., $\bz\in\Bp{\bzz}$), $\bv$ can be expanded as
\begin{eqnarray}
\label{Ls_taylor4_1}
\bv(\bx+\bz)-\bv(\bx)=\nabla\bv(\bx^+)\bz+\br_+(\bv;\bx,\bz),
\end{eqnarray}
where
\begin{eqnarray}
\label{Ls_remainder4_1}
\br_+(\bv;\bx,\bz)&=&\frac{1}{2} \nabla\nabla\bv(\bxi_+)\;\bz\otimes\bz
\end{eqnarray}
for some $\bxi_+$ on the line segment joining $\bx$ and  $\bx+\bz$. Similarly, $\bv$ can be expanded in a Taylor series in the 
$-$side of $\Gamma$. For $\bz\in\Bm{\bzz}$,
\begin{eqnarray}
\label{Ls_taylor4_2}
\bv(\bx+\bz)-\bv(\bx)=\nabla\bv(\bx^-)\bz+\br_-(\bv;\bx,\bz),
\end{eqnarray}
where
\begin{eqnarray}
\label{Ls_remainder4_2}
\br_-(\bv;\bx,\bz)&=&\frac{1}{2} \nabla\nabla\bv(\bxi_-)\;\bz\otimes\bz
\end{eqnarray}
for some $\bxi_-$ on the line segment joining $\bx$ and  $\bx+\bz$.
Substituting \eqref{Ls_taylor4_1} and \eqref{Ls_taylor4_2} in \eqref{Ldel_s_3}, expanding the integrals, 
and rearranging the tensor products, we find
\begin{eqnarray}
\label{Ldel_s_4}
\nonumber
(\Ldel_s \bv)(\bx)&=& \frac{15}{|B_\delta|}(2\mu_+)\int_{\Bp{\bzz}}
   \frac{\bz\otimes\bz\otimes\bz}{|\bz|^4}  \,d\bz\;\nabla\bv(\bx^+)
\nonumber
+ \frac{15}{|B_\delta|}(2\mu_+)\int_{\Bp{\bzz}}
   \frac{\bz\otimes\bz}{|\bz|^4}\;\br_+(\bv;\bx,\bz)\,d\bz\\
\nonumber
&+& \frac{15}{|B_\delta|}(\mu_++\mu_-)\int_{\Bm{\bzz}}
   \frac{\bz\otimes\bz\otimes\bz}{|\bz|^4}  \,d\bz\;\nabla\bv(\bx^-)
 + \frac{15}{|B_\delta|}(\mu_++\mu_-)\int_{\Bm{\bzz}}
   \frac{\bz\otimes\bz}{|\bz|^4}\;\br_-(\bv;\bx,\bz)\,d\bz.\\
\end{eqnarray}
Using \eqref{Ls_remainder4_1}, we obtain the following bound 
\begin{eqnarray}
\label{Ls_estimate3_1}
\nonumber
\left| 
\frac{15}{|B_\delta|}(2\mu_+)\int_{\Bp{\bzz}}
   \frac{\bz\otimes\bz}{|\bz|^4}\;\br_+(\bv;\bx,\bz)\,d\bz\right| 
   &\leq& \frac{15}{|B_\delta|}(2\mu_+)\int_{\Bp{\bzz}}
   \frac{\left|\bz\otimes\bz\otimes\bz\otimes\bz\right|}{|\bz|^4}\,d\bz \left|\frac{1}{2}\nabla\nabla\bv(\bxi_+)\right|\\
&=& \BigO{1}.
\end{eqnarray}
Similarly, one finds 
\begin{eqnarray}
\label{Ls_estimate3_2}
\frac{15}{|B_\delta|}(\mu_++\mu_-)\int_{\Bm{\bzz}}\frac{\bz\otimes\bz}{|\bz|^4}\;\br_-(\bv;\bx,\bz)\,d\bz=\BigO{1},
\end{eqnarray}
and hence the second and fourth terms in \eqref{Ldel_s_4} are finite in the limit as $\delta\rightarrow 0$. 
Using \eqref{Ls_estimate3_1}, \eqref{Ls_estimate3_2}, and using the fact that
\[
0=\int_{B_\delta(\bzz)}
   \frac{\bz\otimes\bz\otimes\bz}{|\bz|^4}  \,d\bz= \int_{\Bp{\bzz}}
   \frac{\bz\otimes\bz\otimes\bz}{|\bz|^4}  \,d\bz+\int_{\Bm{\bzz}}
   \frac{\bz\otimes\bz\otimes\bz}{|\bz|^4}  \,d\bz,
\]
equation \eqref{Ldel_s_4} becomes
\begin{eqnarray}
\label{Ldel_s_5}
\nonumber
(\Ldel_s \bv)(\bx)&=& \frac{15}{|B_\delta|}\int_{\Bp{\bzz}}
   \frac{\bz\otimes\bz\otimes\bz}{|\bz|^4}  \,d\bz\;\left(2\mu_+\nabla\bv(\bx^+)-(\mu_++\mu_-)\nabla\bv(\bx^-)\right)+\BigO{1}.\\
\end{eqnarray}
From Lemma \ref{lem_K_delta} (see Section \ref{interf_model_sec}), the third-order tensor
\begin{equation}
\label{K_delta}
 \mathbb{K}_\delta := \frac{1}{|B_\delta|}\int_{\Bp{\bzz}}\frac{\bz\otimes\bz\otimes\bz}{|\bz|^4}\,d\bz
\end{equation}
behaves, in the limit as $\delta\rightarrow 0$, as
\begin{equation}
\label{K_delta_2}
  \mathbb{K}_\delta \approx \frac{1}{\delta}\;  \mathbb{K}
\end{equation}
for a constant third-order tensor $ \mathbb{K}$.
Thus, equations \eqref{Ldel_s_5}, \eqref{K_delta}, and \eqref{K_delta_2} imply that
\begin{equation}
\label{lim_Ls_interface}
 \limd \left(\Ls\bv\right)(\bx)=\infty\;\;\;\; \mbox{ for } \bx\in\Gamma,
\end{equation}
and, consequently, that the sequence $\left(\Ls\bv\right)_\delta$ is unbounded in $L^p(\Omega)$.
\end{proof}
\begin{remark}
If $\bv$ is smooth at  the interface then \eqref{Ldel_s_5}, in the proof above, becomes
\begin{equation}
 \label{Ls_5}
\left(\Ls\bv\right)(\bx)=
\frac{15}{|B_\delta|}\int_{\Bp{\bzz}}\frac{\bz\otimes\bz\otimes\bz}{|\bz|^4}\,d\bz\;\left(\mu_+-\mu_-\right)\nabla\bv(\bx)
+\BigO{1},
\end{equation}
and hence \eqref{lim_Ls_interface} still hold in this case.
\end{remark}

The behavior of the operator $\Ld$ in \eqref{Ldel_d_3} at material interfaces, in the local limit, is given by the following result.
\begin{lem}
\label{nonconvg_2}
 Assume that $\mu$ and $\lambda$ are given by \eqref{lambda} and that the vector field $\bv$ 
is continuous on $\Omega$ and smooth on $\Omega\setminus\Gamma$.
Then for $\bx\in\Gamma$,
\[
 \limd \left(\Ld\bv\right)(\bx) \mbox{ does not exist}.
\]
Moreover, the sequence $\left(\Ld\bv\right)_\delta$ is unbounded in $L^p(\Omega)$, with $1\leq p<\infty$.
\end{lem}
The proof of this lemma is similar to that of Lemma \ref{nonconvg_1} and thus will not be presented here.
However, we note that for $\bx\in\Gamma$, it can be shown that
\begin{equation}
 \label{Ld_5}
\left(\Ld\bv\right)(\bx)=
\left((\lambda_+ -\mu_+)(\Div\bv)(\bx^+)-(\lambda_- -\mu_-)(\Div\bv)(\bx^-)\right)\;
\frac{3}{|B_\delta|}\int_{\Bp{\bx}}\frac{\by-\bx}{|\by-\bx|^2}\,d\by\;
+\BigO{1},
\end{equation}
and  
\begin{equation}
 \label{Ld_6}
\limd\frac{3\delta}{|B_\delta|}\int_{\Bp{\bx}}\frac{\by-\bx}{|\by-\bx|^2}\,d\by\;=\frac{9}{8}\bn.
\end{equation}

The following result provides a summary  to the behavior of the linear peridynamics operator $\Ldel$, in the limit
as $\delta\rightarrow 0$, in the presence of material interfaces.
\begin{thm}
\label{nonconvg_thm}
Assume that the material properties $\mu$ and $\lambda$ are  smooth on $\Omega\setminus\Gamma$ and have 
jump discontinuities across the interface $\Gamma$. Assume further that the vector field $\bv$ 
is continuous on $\Omega$ and smooth on $\Omega\setminus\Gamma$.
Then 
\begin{enumerate}
\item[(i)] for $\bx\in\Omega\setminus\Gamma$,
\[
 \limd \left(\Ldel\bv\right)(\bx) = \N\bv(\bx),
\]
where $\N$ is the operator of linear elasticity given by \eqref{N}, and
 \item[(ii)] for $\bx\in\Gamma$,
\[
 \limd \left(\Ldel\bv\right)(\bx) \mbox{ does not exist}.
\]
Moreover, the sequence $\left(\Ld\bv\right)_\delta$ is unbounded in $L^p(\Omega)$, with $1\leq p<\infty$.
\end{enumerate}
\end{thm}
\begin{proof}
Part (ii) follows from Lemma \ref{nonconvg_1} and Lemma \ref{nonconvg_2}. For part (i), let 
$\bx\in\Omega\setminus\Gamma$. Then for sufficiently small $\delta$, the ball $B_\delta(\bx)$ is away from the interface. 
Thus, Theorem \ref{thm_convg} applies and part (i) follows.
\end{proof}


\section{A New Peridynamics Model for Material Interfaces}
\label{interf_model_sec}
In this section we introduce a peridynamics model for heterogeneous media  in the presence of material interfaces. Our model
consists of a modified version of the linear peridynamics operator $\Ldel$, given by \eqref{Ldel0}, \eqref{Ldel_s_2}, and \eqref{Ldel_d_3}, 
together with a {\it nonlocal interface condition}. 
This new model is shown, in Theorem \ref{thm_interface_model}, to converge to the classical interface model of linear elasticity.


\subsection{Peridynamics  interface conditions}
\label{interf_model_sec_1}
The divergence of peridynamics at interfaces in the limit of vanishing nonlocality (see Theorem \ref{nonconvg_thm}) is, in fact, not 
surprising.
Indeed, let us consider the corresponding interface problem in linear elasticity inside a two-phase composite, with
$\Omega=\Omega_+\cup\Omega_-\cup\Gamma$ as described in Section \ref{nonconvg_sec}.
The strong form of the elastic equilibrium interface problem is given by the following system of partial differential equations 
\begin{equation}
\label{interface_pde}
 \left\{
\begin{array}{rll}
 \displaystyle \Div\sigma(\bx) &=\bb(\bx),\;\;\; &\bx\in\Omega_+ \\
 \displaystyle \Div\sigma(\bx) &=\bb(\bx),\;\;\; &\bx\in\Omega_- \\
\\
\displaystyle\jump{\sigma\bn}&=0,\;\;\; &\bx\in \Gamma\\
\displaystyle\jump{\bu}&=0,\;\;\; &\bx\in \Gamma
\end{array}
\right.
\end{equation}
where $\sigma$ is the stress tensor given by \eqref{sigma}.
We emphasize that imposing interface jump conditions (the last two equations of \eqref{interface_pde}) is necessary for 
a classical solution $\bu$ defined on $\Omega$ to exist.
Therefore, in order to recover the interface problem in elasticity, given by \eqref{interface_pde} as the local limit of 
peridynamics inside heterogeneous media in the  presence of material interfaces,
we need to introduce a peridynamics interface model and  impose  nonlocal interface conditions 
such that 
the model satisfies  the conditions C(i)-C(iii), introduced in  Section \ref{sec_intro} (Introduction).

We note that Theorem \ref{thm_convg} and Theorem \ref{nonconvg_thm}
imply that if we assume
\begin{equation}
 \label{natural}
(\Ldel\bu)(\bx)=0,\;\;\; \bx\in\Gamma,
\end{equation}
then  
as $\delta\rightarrow 0$,
\begin{equation}
\label{natural_convg}
\Ldel \bu\longrightarrow \N\bu\;\;\;\; \mbox{ in } L^p(\Omega)^3
\end{equation}
for $1\leq p<\infty$. This means that the following system 
\begin{equation*}
\label{natural_interface}
 \left\{
\begin{array}{rll}
 \displaystyle \Ldel\bu(\bx) &=\bb(\bx),\;\;\; &\bx\in\Omega\\
 \\
 \displaystyle \Ldel\bu(\bx) &=0,\;\;\; &\bx\in \Gamma
\end{array}
\right.\tag{\ref{natural_interface_sys}}
\end{equation*}
satisfies conditions C(i) and C(iii), and, 
since the peridynamics operator $\Ldel$ has been kept without modifications 
we call \eqref{natural} {\it peridynamics natural interface condition}.
However, we show in Proposition \ref{prop_wrong_interface_conditions} that 
\eqref{natural} does not satisfy requirement (ii)
since the local limit of 
\eqref{natural} is different from the local interface condition 
\begin{equation}
\label{traction_cond}
 \sigma\bn(\bx^+) =\sigma\bn(\bx^-),\;\;\; \bx\in \Gamma.
\end{equation}

By applying a coordinate translation, we may assume that the unit vector $\bn$ is the normal to the interface  at the origin
(i.e., $\bn=\bn({\bf 0})$). 
\begin{lem}
\label{lem_K_delta}
Let $\mathbb{K}_\delta$ be given by \eqref{K_delta}. Then
\begin{equation}
 \label{K_delta_limit}
\limd \delta\;\mathbb{K}_\delta=\mathbb{K},
\end{equation}
where the third-order tensor $\mathbb{K}$ satisfies
\begin{equation}
 \label{K_A}
\mathbb{K} A = \frac{3}{32} \left(\left(A+A^T\right)\bn + \left(\tr(A)-A\bn\cdot\bn\right)\bn\right)
\end{equation}
for any second-order tensor $A$.
\end{lem}
\begin{proof}
 To emphasize the dependence of the set $\Bp{\bzz}$ on the normal $\bn$, we denote this set by $B_\delta^{\bn+}(\bzz)$.
Using spherical coordinates the unit normal $\bn$ can be represented  by 
\[ 
\bn=\left( 
\begin{array}{c}
\cos{\phi}\sin{\theta}\\
\sin{\phi}\sin{\theta}\\
\cos{\theta}
\end{array} 
\right),
\] 
where $0\leq\phi\leq 2\pi$ and $0\leq\theta\leq \pi$. Define the rotation matrix
\begin{equation}
\label{R} 
R=\left( 
\begin{array}{ccc}
\cos{\phi} \cos{\theta}& \sin{\phi}\cos{\theta} & -\sin{\theta} \\
-\sin{\phi} & \cos{\phi} & 0 \\
\cos{\phi} \sin{\theta}& \sin{\phi}\sin{\theta} & \cos{\theta}  
\end{array} 
\right),
\end{equation}
and notice that 
\[
 R\;\bn=\hat{\bz}_3
=\left( 
\begin{array}{c}
0\\
0\\
1
\end{array} 
\right).
\]
Then, by applying the change of coordinates $\bz=R \bw$, we find
\begin{eqnarray}
 \label{K_delta_chng}
\nonumber
\delta\;\mathbb{K}_\delta &=& \frac{\delta}{|B_\delta|}\int_{B_\delta^{\bn+}(\bzz)}\frac{\bw\otimes\bw\otimes\bw}{|\bw|^4}\,d\bw\\
&=& \frac{\delta}{|B_\delta|}\int_{B_\delta^{\hat{\bz}_3 +}(\bzz)}\frac{R^{-1}\bz\otimes R^{-1}\bz\otimes R^{-1}\bz}{|R^{-1}\bz|^4}
\det(R^{-1})\,d\bz\\
\nonumber
&=& \frac{\delta}{|B_\delta|}\int_{B_\delta^{\hat{\bz}_3 +}(\bzz)}\frac{R^{T}\bz\otimes R^{T}\bz\otimes R^{T}\bz}{|\bz|^4}
\,d\bz,
\end{eqnarray}
where in the last step we have used the facts that $R^{-1}=R^T$, $|R^T \bz|=|\bz|$, and $\det(R^{-1})=1$.
Since the interface $\Gamma$ is smooth, we may assume that, in the limit as $\delta\rightarrow 0$, the set $B_\delta^{\hat{\bz}_3 +}$
is a half ball. Thus, a straightforward calculation using spherical coordinates shows that
\[
\limd \delta\;\mathbb{K}_\delta=\mathbb{K},
\]
where the entries $\mathbb{K}_{i j k}$ of the third-order tensor $\mathbb{K}$  are given by
\begin{equation}
\label{K_entries}
\left\{
\begin{array}{ccl}
 \mathbb{K}_{1 1 1} &=& \frac{3}{32} \cos{\phi} \;\sin{\theta}\left(3-\cos^2{\phi} \;\sin^2{\theta}\right),\\
 \mathbb{K}_{1 1 2} &=& \frac{3}{32} \sin{\phi} \;\sin{\theta}\left(1-\cos^2{\phi} \;\sin^2{\theta}\right)
=\mathbb{K}_{1 2 1}=\mathbb{K}_{2 1 1},\\
 \mathbb{K}_{1 1 3} &=& \frac{3}{32} \cos{\theta}\left(1-\cos^2{\phi} \;\sin^2{\theta}\right)
=\mathbb{K}_{1 3 1}=\mathbb{K}_{3 1 1},\\
 \mathbb{K}_{1 2 2} &=& \frac{3}{32} \cos{\phi} \;\sin{\theta}\left(1-\sin^2{\phi} \;\sin^2{\theta}\right)
=\mathbb{K}_{2 1 2}=\mathbb{K}_{2 2 1},\\
 \mathbb{K}_{1 2 3} &=& -\frac{3}{32} \sin{\phi}\;\cos{\phi} \;\sin^2{\theta}\;\cos{\theta}
=\mathbb{K}_{1 3 2}=\mathbb{K}_{2 1 3}=\mathbb{K}_{2 3 1}=\mathbb{K}_{3 1 2}=\mathbb{K}_{3 2 1},\\
 \mathbb{K}_{1 3 3} &=& \frac{3}{32} \cos{\phi} \;\sin^3{\theta}
=\mathbb{K}_{3 1 3}=\mathbb{K}_{3 3 1},\\
 \mathbb{K}_{2 2 3} &=& \frac{3}{32} \cos{\theta} \left(1-\sin^2{\phi} \;\sin^2{\theta}\right)
=\mathbb{K}_{2 3 2}=\mathbb{K}_{3 2 2},\\
 \mathbb{K}_{2 3 3} &=& \frac{3}{32} \sin{\phi} \;\sin^3{\theta}
=\mathbb{K}_{3 2 3}=\mathbb{K}_{3 3 2},\\
 \mathbb{K}_{2 2 2} &=& \frac{3}{32} \sin{\phi}\;\sin{\theta} \left(3-\sin^2{\phi} \;\sin^2{\theta}\right),\\
 \mathbb{K}_{3 3 3} &=& \frac{3}{32} \cos{\theta} \left(3-\cos^2{\theta}\right).
\end{array}
\right.
\end{equation}
By calculating $\mathbb{K} A$, using \eqref{K_entries}, and comparing it with
$\frac{3}{32} \left(\left(A+A^T\right)\bn + \left(\tr(A)-A\bn\cdot\bn\right)\bn\right)$, one finds that \eqref{K_A} holds true.
\end{proof}
Applying Lemma \ref{lem_K_delta} with $A=\nabla\bv$ we obtain the following result.
\begin{cor}
\label{cor_K_delta}
Assume that $\bv$ is a differentiable vector-field. Then
\begin{equation}
 \label{K_delta_grad_v}
\limd \frac{\delta}{|B_\delta|}\int_{\Bp{\bzz}}\frac{\bz\otimes\bz\otimes\bz}{|\bz|^4}\,d\bz\;\nabla\bv= 
\frac{3}{32} \left( \left(\nabla\bv+\nabla\bv^T\right)\bn + (\Div\bv)\;\bn-\left(\nabla\bv\;\bn\cdot\bn\right)\bn\right).
\end{equation}
\end{cor}
The local limit of  peridynamics' natural interface condition \eqref{natural} is given by the following result.
\begin{prop}
\label{prop_wrong_interface_conditions}
Assume that the material properties $\mu$ and $\lambda$ are given by \eqref{lambda}.
Assume further that $\bu$ 
is continuous on $\Omega$ and smooth on $\Omega\setminus\Gamma$.
Then for $\bx\in\Gamma$
\begin{eqnarray}
 \label{local_of_perid_natural}
\nonumber
\limd \delta(\Ldel \bu)(\bx) &=& \frac{45}{32} \bigg(\jump{(\mu_++\mu)(\nabla\bu+\nabla\bu^T)}\bn+
\jump{(\mu_++\mu)\Div\bu}\bn \\
&&\;\;\;\;\;\;\;\;\;\; - \jump{(\mu_++\mu)\nabla\bu}\bn\cdot\bn\;\bn + \frac{4}{5}\jump{(\lambda-\mu)\Div\bu}\bn\bigg).
\end{eqnarray}

\end{prop}
\begin{proof}
Let $\bx\in\Gamma$. Using Corollary \ref{cor_K_delta} and from  \eqref{Ldel_s_5} 
, one finds that
\begin{eqnarray}
\label{L_s_6}
\nonumber
\limd\delta(\Ldel_s \bv)(\bx)&=& \frac{45}{32} \bigg(\jump{(\mu_++\mu)(\nabla\bu+\nabla\bu^T)}\bn+
\jump{(\mu_++\mu)\Div\bu}\bn
- \jump{(\mu_++\mu)\nabla\bu}\bn\cdot\bn\;\bn\bigg).\\
\end{eqnarray}
And from \eqref{Ld_5} and \eqref{Ld_6}, one finds that
\begin{eqnarray}
 \label{Ld_7}
\limd\delta\left(\Ld\bu\right)(\bx)=
 \frac{9}{8}\jump{(\lambda-\mu)\Div\bu}\bn.
\end{eqnarray}
Equation \eqref{local_of_perid_natural} follows from \eqref{L_s_6} and \eqref{Ld_7}.
\end{proof}
\begin{remark}
Note that for $\bx\in\Gamma$ and $\sigma$ given by \eqref{sigma} then
\begin{eqnarray}
\label{sigma_jump}
 \jump{\sigma}\bn=\jump{\lambda\Div\bu}\bn+\jump{\mu(\nabla\bu+\nabla\bu^T}\bn.
\end{eqnarray}
Comparing \eqref{sigma_jump} and \eqref{local_of_perid_natural} we conclude that 
the local interface condition \eqref{traction_cond} is not recoverable from the nonlocal interface condition \eqref{natural}.
\end{remark}


\begin{figure}[t]

\centering
    \includegraphics[width=.3\textwidth]{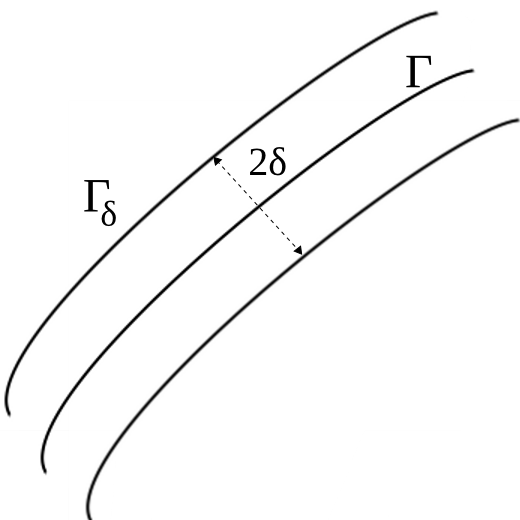}
\caption{\label{fig_gamma_delta}The extended interface $\Gamma_\delta$.}   
\end{figure}


\subsection{A peridynamic interface model}
\label{sec_perid_interface_model}
Let $\Gamma_\delta$ be the set defined by
\[
 \Gamma_\delta=\{\bx\in\Omega:|\bx-\Gamma|<\delta\},
\]
where $|\bx-\Gamma|$ denotes the distance between the point $\bx$ and the interface. We refer to this three-dimensional set 
as the {\it extended interface}. An illustration of this set is shown in Figure \ref{fig_gamma_delta}.

The peridynamics material  interface model, under conditions of equilibrium, is given by
\begin{equation}
\label{interface_model}
 \left\{
\begin{array}{rll}
 \displaystyle \Ldel_*\bu(\bx) &=\bb(\bx),\;\;\; &\bx\in\Omega\\
 \\
 \displaystyle \Ldel_*\bu(\bx) &=0,\;\;\; &\bx\in \Gamma_\delta
\end{array},
\right.
\end{equation}
where
\begin{eqnarray}
\label{L*}
\Ldel_*\bu=\Ldel\bu+1_{\Gamma_\delta}\, \Ldel_{\Gamma_\delta}\bu.
\end{eqnarray}
Here $\Ldel$ is  given by \eqref{Ldel0}, \eqref{Ldel_s_2}, and \eqref{Ldel_d_3}, $1_{\Gamma_\delta}$ is the indicator function
\begin{equation*}
 1_{\Gamma_\delta}(\bx)=\left\{
\begin{array}{ll}
 \displaystyle 1,& \bx\in\Gamma_\delta\\
 \displaystyle 0,& \bx\not\in\Gamma_\delta
 \end{array},
\right.
\end{equation*}
and the operator $\Ldel_{\Gamma_\delta}$ is defined by
\begin{eqnarray}
\label{L_Gamma}
\nonumber
\Ldel_{\Gamma_\delta}\bu(\bx)&=& -\frac{15}{|B_\delta|}\int_{B_\delta(\bx)}
  \mu(\bx) \frac{(\by-\bx)
\otimes(\by-\bx)}{|\by-\bx|^4}\big( \bu(\by)-\bu(\bx)\big)
  \,d\by,\\
  \nonumber
  && +\frac{1}{4}\frac{9}{|B_\delta|^2}\int_{B_\delta(\bx)}\int_{B_\delta(\by)}
 \left(\lambda(\by)-\mu(\by)\right)  \frac{\by-\bx}{|\by-\bx|^2}
\otimes\frac{\bz-\by}{|\bz-\by|^2}\,\bu(\bz)
 \,d\bz d\by\\
& & +\frac{5}{4} \frac{9}{|B_\delta|^2}\int_{B_\delta(\bx)}\int_{B_\delta(\by)}
 \mu(\by)  \frac{\by-\bx}{|\by-\bx|^2}
\cdot\frac{\bz-\by}{|\bz-\by|^2}\,\bu(\bz)
 \,d\bz d\by \cdot \bn(\bx)\;\;\bn(\bx).
\end{eqnarray}

Similarly, the general peridynamics material  interface model is given by
 \begin{equation}
\label{interface_model_dynamics}
 \left\{
\begin{array}{rll}
 \displaystyle \rho(\bx) \ddot{\bu}(\bx,t) &= \Ldel_* \bu(\bx) + \bb(\bx,t),
 \;\;\; &\bx\in\Omega\\
 \\
 \displaystyle \Ldel_*\bu(\bx) &=0,\;\;\; &\bx\in \Gamma_\delta
\end{array}.
\right.
\end{equation}

The convergence of the peridynamics interface model \eqref{interface_model} to the local interface model \eqref{interface_pde} is given
by the next result.
\begin{thm}
\label{thm_interface_model}
 Assume that $\mu$ and $\lambda$ are given by \eqref{lambda} and that the vector field $\bu$ 
is continuous on $\Omega$ and smooth on $\Omega\setminus\Gamma$. Then if
\begin{equation}
\label{nonlocal_inteface_cond}
\Ldel_*\bu(\bx) =0,\;\;\; \mbox{for all } \;\bx\in \Gamma_\delta,
\end{equation}
then as $\delta \rightarrow 0$,
\begin{enumerate}
\item 
\begin{equation}
\label{lim_Ldel*}
\Ldel_* \bu\longrightarrow \N\bu,\;\;\;\; \mbox{ in } L^p(\Omega)^3,\;\;\;\mbox{ for } 1\leq p<\infty, \mbox{ and }
\end{equation}
\item 
\begin{equation}
\label{jump_cond_1}
\jump{\sigma}\bn=0, \;\;\;\mbox{ for } \bx\in\Gamma.
\end{equation}
\end{enumerate}

\end{thm}
\begin{remark}
\begin{itemize}
 \item The first part of Theorem \ref{thm_interface_model} shows that imposing the nonlocal interface condition \eqref{nonlocal_inteface_cond} 
implies that
the Navier operator $\N$ is the local limit of the operator $\Ldel_*$ and, consequently, the model \eqref{interface_model}
satisfies C(i).
\item The second part of Theorem \ref{thm_interface_model} shows that the local interface condition \eqref{jump_cond_1}
can be recovered from the local limit of the
nonlocal interface condition \eqref{nonlocal_inteface_cond} and, consequently, the model \eqref{interface_model}
satisfies C(ii).
\item The peridynamics interface model \eqref{interface_model} satisfies C(iii).
\end{itemize}

\end{remark}

\begin{proof}
\noindent {\bf Part (1)}. We show that \eqref{nonlocal_inteface_cond} implies \eqref{lim_Ldel*}.

Let $\bx\in\Omega$ with a distance of at least $2\delta$ from $\partial \Omega$. Then 
for $\bx\not\in\Gamma_\delta$, and by using \eqref{L*} we obtain 
\begin{equation}
\label{L*isL}
\Ldel_*\bu(\bx)=\Ldel\bu(\bx).
\end{equation}
From \eqref{L*isL} and by using Theorem \ref{thm_convg}, one finds
\begin{equation}
\label{limL*1}
\limd\Ldel_*\bu(\bx)=\N(\bx).
\end{equation}
On the other hand, for $\bx\in\Gamma_\delta$, and by using the assumption \eqref{nonlocal_inteface_cond}, one finds that
\begin{equation}
\label{limL*2}
\limd \Ldel_*\bu(\bx)=0.
\end{equation}
Since $\Gamma_\delta\rightarrow \Gamma$ as $\delta\rightarrow 0$ and that
$|\Gamma|=0$, it follows from  \eqref{limL*1} and \eqref{limL*2} that
\begin{equation}
\label{limL*3}
\limd \Ldel_*\bu(\bx)=\N(\bx), \mbox{ for \it{almost every }} \bx\in \Omega.
\end{equation}
Using \eqref{limL*3} and Lebesgue's dominated convergence theorem, \eqref{lim_Ldel*} follows.\\

\noindent {\bf Part (2)}. We show that \eqref{nonlocal_inteface_cond} implies \eqref{jump_cond_1}.

Let $\bx\in\Gamma$. Then,  by multiplying both sides of \eqref{nonlocal_inteface_cond} by $\delta$ and taking the limit, one obtains
\begin{equation}
\label{limdL*1}
\limd \delta\Ldel_*\bu(\bx)=0.
\end{equation}
Next, we show that
\begin{eqnarray}
 \label{limdL*2}
\limd \delta\Ldel_* \bu(\bx) &=& \frac{45}{32} \bigg(\jump{\lambda\Div\bu}\bn+\jump{\mu(\nabla\bu+\nabla\bu^T}\bn\bigg)\\
\label{limdL*3}
 &=& \frac{45}{32} \jump{\sigma}\bn.
\end{eqnarray}
Equation \eqref{jump_cond_1} follows from 
\eqref{limdL*1} and \eqref{limdL*3}. Thus, it remains to prove \eqref{limdL*2} to complete the proof.

From \eqref{L*} and since $\bx\in\Gamma$, $\Ldel_*\bu(\bx)$ can be written as
\begin{eqnarray}
\label{L*decompose}
\nonumber
\Ldel_*\bu(\bx)&=& \Ldel\bu(\bx)+\Ldel_{\Gamma_\delta}\bu(\bx)\\
\nonumber
&=& \left(\Ls\bu(\bx)+\Ld\bu(\bx)\right)+\left(\Ldel_1\bu(\bx)+\frac{1}{4}\Ld\bu(\bx)+\Ldel_2\bu(\bx)\right)\\
&=& \Ls\bu(\bx)+\frac{5}{4}\Ld\bu(\bx)+\Ldel_1\bu(\bx)+\Ldel_2\bu(\bx),
\end{eqnarray}
where
\begin{eqnarray}
\label{L*decompose2}
\Ldel_1\bu(\bx)&=& -\frac{15}{|B_\delta|}\int_{B_\delta(\bx)}
  \mu(\bx) \frac{(\by-\bx)
\otimes(\by-\bx)}{|\by-\bx|^4}\big( \bu(\by)-\bu(\bx)\big)
  \,d\by,\\
\nonumber\\
\label{L*decompose3}
\Ldel_2\bu(\bx)&=& 
\frac{5}{4} \frac{9}{|B_\delta|^2}\int_{B_\delta(\bx)}\int_{B_\delta(\by)}
 \mu(\by)  \frac{\by-\bx}{|\by-\bx|^2}
\cdot\frac{\bz-\by}{|\bz-\by|^2}\,\bu(\bz)
 \,d\bz d\by \cdot\bn(\bx)\;\;\bn(\bx).
\end{eqnarray}
We note that the definition of $\Ldel_1$ is similar to that of $\Ldel_s$. Thus, and since $\mu$ is given by \eqref{lambda}, an argument similar to the derivation of \eqref{Ldel_s_5} in Lemma \ref{nonconvg_1} yields
\begin{eqnarray}
\label{Ldel_1_1}
(\Ldel_1 \bu)(\bx)&=& -\frac{15}{|B_\delta|}\int_{\Bp{\bzz}}
   \frac{\bz\otimes\bz\otimes\bz}{|\bz|^4}  \,d\bz\;\;\mu_+\left(\nabla\bu(\bx^+)-\nabla\bu(\bx^-\right)+\BigO{1}.
\end{eqnarray}
From  \eqref{Ldel_1_1} and \eqref{K_delta_grad_v}, one finds that
\begin{eqnarray}
\label{L_1_2}
\limd\delta(\Ldel_1 \bu)(\bx)&=& \frac{45}{32} \;\mu_+\;\bigg(\jump{\nabla\bu+\nabla\bu^T}\bn+
\jump{\Div\bu}\bn
 - \jump{\nabla\bu}\bn\cdot\bn\;\bn\bigg).
\end{eqnarray}
Combining \eqref{L_1_2}, \eqref{L_s_6}, and \eqref{Ld_7}, we obtain
\begin{eqnarray}
\label{L_sd1}
\nonumber
\limd\delta\left(\Ldel_s+\frac{5}{4}\Ldel_d+\Ldel_1\right)(\bu)(\bx)&=& \frac{45}{32} \bigg(\jump{\lambda\Div\bu}\bn+\jump{\mu(\nabla\bu+\nabla\bu^T}\bn
- \jump{\nabla\bu}\bn\cdot\bn\;\bn\bigg).\\
\end{eqnarray}
Equation \eqref{limdL*2} follows from \eqref{L_sd1}, \eqref{L*decompose3}, \eqref{L*decompose}, and Lemma \ref{nnn_lemma}, completing the proof 
of  Part (2).
\end{proof}

\begin{lem}
\label{nnn_lemma}
Let $\bx$ be a point on the interface $\Gamma$ and $\bn=\bn(\bx)$ be the unit normal to the interface. Assume that the vector field $\bv$ is smooth on $\Omega\setminus\Gamma$. Then 
\begin{eqnarray}
\label{nnn}
\limd\delta(\Ldel_2 \bv)(\bx)&=& \frac{45}{32}\jump{\nabla\bv}\bn\cdot\bn\;\bn.
\end{eqnarray}
\end{lem}

\begin{proof}
It is sufficient to show that
\begin{eqnarray}
\label{n_1}
\limd\frac{3\delta}{|B_\delta|^2}\int_{B_\delta(\bx)}\int_{B_\delta(\by)}
 \mu(\by)  \frac{\by-\bx}{|\by-\bx|^2}
\cdot\frac{\bz-\by}{|\bz-\by|^2}\,\bv(\bz)
 \,d\bz d\by&=& \frac{3}{8}\jump{\nabla\bv}\bn.
\end{eqnarray}
The Taylor expansion of $\bv$ about $\bz=\by$ is given by
\begin{equation}
\label{L2_taylor1}
\bv(\bz)=\bv(\by)+\nabla\bv(\by) \;(\bz-\by)+\br(\bv;\bz,\by),
\end{equation}
where
\begin{equation}
\label{L2_remainder1}
\br(\bv;\bz,\by)=\frac{1}{2} \nabla\nabla\bv(\bxi)\, \left((\bz-\by)\otimes(\bz-\by)\right)
\end{equation}
for some $\bxi$ on the line segment joining $\bz$ and $\by$. Using \eqref{L2_taylor1}, the integral on the left hand side of \eqref{n_1} becomes
\begin{eqnarray}
\label{n_2}
\nonumber
\frac{3\delta}{|B_\delta|^2}\int_{B_\delta(\bx)}\int_{B_\delta(\by)}
 \mu(\by)  \frac{\by-\bx}{|\by-\bx|^2}
\cdot\frac{\bz-\by}{|\bz-\by|^2}\,\bv(\bz)
 \,d\bz d\by&&\\
 \nonumber
\\
 \nonumber
&&\hspace*{-3cm}= \frac{3\delta}{|B_\delta|^2}
 \int_{B_\delta(\bx)}\mu(\by)  \frac{\by-\bx}{|\by-\bx|^2}\cdot\int_{B_\delta(\by)}
\frac{\bz-\by}{|\bz-\by|^2}\,d\bz \;\bv(\by)
  d\by\\
  \nonumber
\\
 \nonumber
 &&\hspace*{-3cm}+ \frac{3\delta}{|B_\delta|^2}
 \int_{B_\delta(\bx)}\mu(\by)  \nabla\bv(\by)\int_{B_\delta(\by)}
\frac{(\bz-\by)\otimes(\bz-\by)}{|\bz-\by|^2}\,d\bz \;\frac{\by-\bx}{|\by-\bx|^2}
  d\by\\
  \nonumber
\\
 \nonumber
 &&\hspace*{-3cm}+ \frac{3\delta}{|B_\delta|^2}
 \int_{B_\delta(\bx)}\int_{B_\delta(\by)}\mu(\by)  \frac{\by-\bx}{|\by-\bx|^2}\cdot
\frac{\bz-\by}{|\bz-\by|^2}\,\br(\bv;\bz,\by)\,d\bz
  d\by.\\
\end{eqnarray}
We note that by using \eqref{sym}, the first term on the right hand side of \eqref{n_2} is equal to zero and
it is straightforward to show that the third term on the right hand side of \eqref{n_2} is $\BigO{\delta}$.
Thus, by using \eqref{identity_identity} in the second term, equation \eqref{n_2} becomes
\begin{eqnarray}
\label{n_3}
\nonumber
\hspace*{-.5cm}\frac{3\delta}{|B_\delta|^2}\int_{B_\delta(\bx)}\int_{B_\delta(\by)}
 \mu(\by)  \frac{\by-\bx}{|\by-\bx|^2}
\cdot\frac{\bz-\by}{|\bz-\by|^2}\,\bv(\bz)
 \,d\bz d\by 
 &=& \frac{\delta}{|B_\delta|}
 \int_{B_\delta(\bx)}\mu(\by)  \nabla\bv(\by)\;\frac{\by-\bx}{|\by-\bx|^2}
  d\by
  + \BigO{\delta}\\
\nonumber
 &=& \frac{\delta}{|B_\delta|}
 \int_{\Bp{\bx}}\mu(\by)  \nabla\bv(\by)\;\frac{\by-\bx}{|\by-\bx|^2}
  d\by\\
  \nonumber
  &&+\frac{\delta}{|B_\delta|}
 \int_{\Bm{\bx}}\mu(\by)  \nabla\bv(\by)\;\frac{\by-\bx}{|\by-\bx|^2}
  d\by  + \BigO{\delta}.\\  
\end{eqnarray}
Since $\bv$ is smooth on each side of $\Gamma$, then for $\by$ on the $+$side of $\Gamma$ (i.e., $\by\in\Bp{\bx}$), $\nabla\bv$ can be expanded as
\begin{eqnarray}
\label{L2_taylor_1}
\nabla\bv(\by)=\nabla\bv(\bx)+\bR_+(\nabla\bv;\bx,\by),
\end{eqnarray}
where
\begin{eqnarray}
\label{L2_remainder_1}
\bR_+(\nabla\bv;\bx,\by)&=&\nabla\nabla\bv(\bxi_+)\;(\by-\bx)
\end{eqnarray}
for some $\bxi_+$ on the line segment joining $\bx$ and  $\by$. Similarly, $\nabla\bv$ can be expanded  on the 
$-$side of $\Gamma$. For $\by\in\Bm{\bx}$,
\begin{eqnarray}
\label{L2_taylor_2}
\nabla\bv(\by)=\nabla\bv(\bx)+\bR_-(\nabla\bv;\bx,\by),
\end{eqnarray}
where
\begin{eqnarray}
\label{L2_remainder_2}
\bR_-(\nabla\bv;\bx,\by)&=&\nabla\nabla\bv(\bxi_-)\;(\by-\bx)
\end{eqnarray}
for some $\bxi_-$ on the line segment joining $\bx$ and  $\by$.
By substituting \eqref{L2_taylor_1} and \eqref{L2_taylor_2} on the right hand side of \eqref{n_3} and expanding the integrals, we find
\begin{eqnarray}
\label{n_4}
\nonumber
\frac{3\delta}{|B_\delta|^2}\int_{B_\delta(\bx)}\int_{B_\delta(\by)}
 \mu(\by)  \frac{\by-\bx}{|\by-\bx|^2}
\cdot\frac{\bz-\by}{|\bz-\by|^2}\,\bv(\bz)
 \,d\bz d\by 
&=&\frac{\delta}{|B_\delta|}
 \int_{\Bp{\bx}}\mu_+  \nabla\bv(\bx^+)\;\frac{\by-\bx}{|\by-\bx|^2}
  d\by\\
  \nonumber
&&+\frac{\delta}{|B_\delta|}
 \int_{\Bm{\bx}}\mu_-  \nabla\bv(\bx^-)\;\frac{\by-\bx}{|\by-\bx|^2}
  d\by\\
  \nonumber
&&+\frac{\delta}{|B_\delta|}
 \int_{\Bp{\bx}}\mu_+  \bR_+(\nabla\bv;\bx,\by)\;\frac{\by-\bx}{|\by-\bx|^2}
  d\by\\
  \nonumber
 && +\frac{\delta}{|B_\delta|}
 \int_{\Bm{\bx}}\mu_-  \bR_-(\nabla\bv;\bx,\by)\;\frac{\by-\bx}{|\by-\bx|^2}
  d\by\\
  \nonumber
  &&+\BigO{\delta}.\\
\end{eqnarray}
Using \eqref{L2_remainder_1} and \eqref{L2_remainder_2} one can easily show that the third and the fourth terms on the right hand side of \eqref{n_4} are $\BigO{\delta}$, and using \eqref{sym} for the first and second terms on the right hand side of \eqref{n_4}, we find
\begin{eqnarray}
\label{n_5}
\nonumber
\frac{3\delta}{|B_\delta|^2}\int_{B_\delta(\bx)}\int_{B_\delta(\by)}
 \mu(\by)  \frac{\by-\bx}{|\by-\bx|^2}
\cdot\frac{\bz-\by}{|\bz-\by|^2}\,\bv(\bz)
 \,d\bz d\by &&
 \\
 \nonumber
\\
\nonumber
 &&\hspace*{-3cm}
=
 (\mu_+  \nabla\bv(\bx^+)-\mu_-  \nabla\bv(\bx^-))\;\frac{\delta}{|B_\delta|}\;\int_{\Bp{\bx}}\frac{\by-\bx}{|\by-\bx|^2}
  d\by+\BigO{\delta}.\\
\end{eqnarray}
Equation \eqref{n_1} follows from \eqref{n_5} and  \eqref{Ld_6}, completing the proof.
\end{proof}

We conclude this section by providing a mechanical
interpretatation to \eqref{L*} and \eqref{L_Gamma}. Equation \eqref{limdL*3} in the proof of Theorem~\ref{thm_interface_model} provides an important relationship between the jump in the local traction across the interface and the nonlocal operator $\Ldel_*$. This implies that, for points $\bx\in\Gamma$, the expression $\frac{32}{45}\delta \Ldel_*\bu(\bx)$ represents the nonlocal analogue of $\jump{\sigma}\bn$. Therefore, we can interpret $\frac{32}{45}\delta \Ldel_*\bu(\bx)$ as the jump in the nonlocal traction across the interface. 
Moreover, the operator $\Ldel_{\Gamma_\delta}$,  given by \eqref{L_Gamma}, can be interpreted as the missing term in peridynamics which modifies the jump in the nonlocal traction such  that \eqref{limdL*3} holds true. Furthermore,  Theorem~\ref{thm_interface_model} and \eqref{limdL*3} imply that the nonlocal interface condition \eqref{interface_model_sys_2} is the nonlocal analogue of the local interface condition \eqref{interface_pde0_3} and that the peridynamic interface model given by \eqref{interface_model_sys} is the nonlocal analogue of the local interface model given by \eqref{interface_pde0}.

\bibliographystyle{plain}
\bibliography{Perid_interfaces}    

\end{document}